\documentclass[11pt]{article}
\usepackage{amsmath,amssymb,amsthm}
\usepackage{fullpage,parskip}

\theoremstyle{plain}
\newtheorem{theorem}{Theorem}[section]
\newtheorem{lemma}[theorem]{Lemma}

\newtheorem{obs}[theorem]{Observation}

\theoremstyle{definition}
\newtheorem{defi}[theorem]{Definition}

\title{Pixelating relations and functions without adding substructures\thanks{Research supported in part by an Israel Science Foundation grant number 879/22.}}
\author{Eldar Fischer\thanks{Technion -- Israel Institute of Technology. \tt eldar@cs.technion.ac.il}}

\date{}

\begin{document}

\maketitle
\begin{abstract}\noindent
We investigate models of relations over a bounded continuous segment of real numbers, along with the natural linear order over the reals being provided as a ``hard-coded'' relation. This paper presents a generalization of a lemma from \cite{orderon}, showing that with a small amount of modification (measured in terms of the Lebesgue measure) we can replace such a model with a ``pixelated'' one that has a finite description, in a way that preserves all universally quantified statements over the relations, or in other words, without adding any new substructures.
\end{abstract}

\section{Introduction}\label{sec:intro}

Consider a vocabulary with one relation of arity $d$ (later we discuss generalizations to more relations), along with a binary relation ``$\leq$'' which is promised to be a linear order. Moreover, the order relation should be thought of as ``hard-coded'', that is, determined in advance along with the universe $U$. We will be concerned with two specific examples. The first one is the case where $U=[n]$, with $\leq$ being the natural order among the integers. The second case is where $U=\{x\in\mathbb{R}:0<x\leq 1\}$ is a finite segment of real numbers, with $\leq$ being the natural order among them (for the specific investigation here it will not matter much whether $0$ and/or $1$ themselves are contained in $U$).

We investigate the question of approximating continuous models over a segment of real numbers by discrete (finite universe) models. The notion of approximation here is tied with the Lebesgue measure, or equivalently the probability of finding a difference when uniformly drawing a tuple $(x_1,\ldots,x_d)\in U$. To be able to compare a continuous model with a discrete model over $[n]$, one would think of ``stretching'' the latter so that $i\in [n]$ will be identified with the set $\{x\in\mathbb R:(i-1)/n<x\leq i/n\}$.

However, the above idea has a problem, exemplified best by the result of drawing $x$ and $y$ from $U$. In the continuous case, there is zero probability for $x=y$, while for $U=[n]$ this probability is positive. For this reason we will use a slightly more intricate definition of a discrete approximation, one that still has a finite description, while allowing for the additional variations according to the order between $x$ and $y$ along the ``diagonal''. This is formalized in Definition \ref{def:phomo} below.

We must also define what must be preserved when moving to the approximation. Here we will preserve universally quantified sentences. This means that when moving from a model $M$ to a new model $M'$, there will be no sub-model of $M'$ that does not appear already in $M$. In fact, there will be no sub-models of $M'$ apart from those that appear with {\em positive probability} in $M$, when uniformly drawing the elements $x_1<\cdots<x_n$ from $U$ that compose the universe of the sub-model.

This approximation is what we call the ``pixelated version'' of the model for the relation. We present here a generalization of a lemma from \cite{orderon}, which was originally used there for dealing with limit objects of vertex-ordered simple graphs. We show here that for any fixed parameter $\epsilon$, a relation can be $\epsilon$-approximated by a pixelated version, namely an $l$-part homogeneous relation (as per Definition \ref{def:phomo} below). The ``fineness parameter'' $l$ can depend on the original relation itself, which is to be expected, especially that the result holds for sub-models of all finite orders $n$ at once.

For logical concepts, we refer the reader to \cite{ebbinghaus}. However, we will use surprisingly little logic in the following. Essentially we will use an ``encoding'' of the logical models using functions. For example, a model for a vocabulary that includes $r$ relations of arity $d$ over the universe $[n]$, along with the hard-coded natural order $\leq$ over $[n]$, can be represented as $r$ subsets $R_1,\ldots,R_r\subseteq [n]^d$. However, it can also be encoded using a single function $R:[n]^d\to [k]$ for $k=2^r$, where $R(i_1,\ldots,i_d)$ provides all information about which of the $R_i$ contain the tuple $(i_1,\ldots,i_d)$. We will work primarily with the later representation.

The proof here follows the general structure of the proof of the special-case lemma from \cite{orderon}, with multiple adaptations for the generalization. Some basic concepts, in particular from measure theory, are presented in Section \ref{sec:basic}. For completeness we supply proofs for the very basic lemmas that we use from measure theory. After some additional necessary definitions, we present the main result, Theorem \ref{thm:pix}, in Section \ref{sec:present}. We also demonstrate that the bulk of the work is mandated by the ``no new substructures'' requirement, by proving a very simple lemma that omits this requirement.

Section \ref{sec:homo} contains more definitions that are required for the proof mechanism. In Section \ref{sec:ramsey} we provide a proof of a Ramsey type lemma that refers to a scenario involving hypergraphs with vertices of multiple ``sorts'', along with an adaptation for the main proof. This Ramsey-type lemma is much more intricate than the lemma that is used in \cite{orderon}. This is caused by the combination of the move from graphs to higher arity relations with the requirement of selecting sufficiently many vertices of every sort (which is much easier to satisfy for the binary relation case).

Section \ref{sec:proof} finally provides the proof of Theorem \ref{thm:pix}. It is derived from the simple lemma from Section \ref{sec:present}, using a key lemma about the positive probability appearance ``discrete versions'' of pixelated functions in the function to be approximated.

The final Section \ref{sec:disc} contains a discussion of Theorem \ref{thm:pix} and some closely related (and easy to derive) variants, touching also on the role that its special case plays in \cite{orderon} for proving a removal lemma for vertex-ordered graphs using limit objects.

\section{The basics}\label{sec:basic}

Let $\mathbb{I}$ denote the interval of real numbers $\{x\in\mathbb{R}:0<x\leq 1\}$, and for any natural number $k$ let $[k]$ denote the set of integers $\{1,\ldots,k\}$. It will make no essential difference to exclude $0$ from $\mathbb{I}$, since we are concerned with the natural linear order along with the Lebesgue measure (with $\{0\}$ being a set of measure zero), and we will deal only with universally quantified sentences (and having a minimal element is not thus expressible). Not including $0$ in $\mathbb{I}$ will simplify notation later on. We also define for any $l\in\mathbb{N}$ the disjoint subintervals $\mathbb{I}_{i,l}=\{x\in\mathbb{I}:(i-1)/l< x\leq i/l\}$ where $i\in [l]$.

The notion of Lebesgue-measurable sets and functions over $\mathbb{I}^d$ plays a major role here. We refer the reader to \cite{frankslebesgue} for a primer on measure theory, although we will only use here the most basic definitions, and provide self-contained proofs for all else.

We use $\lambda(A)$ to denote the Lebesgue measure of a measurable set $A$. For two sets $A$ and $B$ we denote by $A\Delta B$ their symmetric difference. In particular we have the following well known and almost immediate consequence of the definition of measurability.

\begin{lemma}\label{lem:onepix}
For every measurable set $A\subseteq \mathbb{I}^d$ and $\epsilon$ there exists $l$, so that for every $l'\geq l$ there is a set $B$ consisting of the (disjoint) union of sets of the type $\prod_{j=1}^d \mathbb{I}_{i_j,l'}$, so that $\lambda(B\Delta A)\leq\epsilon$. 
\end{lemma}

\begin{proof}
By definition of the measurability of $A$, there exists a set $B_0$ such that $A\subset B_0$, where $B_0$ is the union of countably many (not necessarily disjoint) ``boxes'' $C_1,C_2,\ldots$, with every $C_j$ being of the form $\{(x_1,\ldots,x_d)\in\mathbb{I}^d:\bigwedge_{i=1}^d\alpha_{j,i}\leq x_i\leq \beta_{j,i}\}$, and such that the sum of measures $\sum_j\prod_{i=1}^d(\beta_{j,i}-\alpha_{j,i})$ is at most $\lambda(A)+\epsilon/3$. Now we set $t$ so that $\sum_{j>t}\prod_{i=1}^d(\beta_{j,i}-\alpha_{j,i})\leq \epsilon/3$ (such a $t$ exists because the whole sum is finite), and set $B_1=\bigcup_{j\leq t}C_j$. Clearly $\lambda(A\setminus B_1)\leq\epsilon/3$ (since $B_0$ contains $A$ and $B_0\setminus B_1\subseteq\bigcup_{j>t}C_j$), and also $\lambda(B_1\setminus A)\leq\epsilon/3$ (since $\lambda(B_0\setminus A)\leq\epsilon/3)$, hence $\lambda(A\Delta B_1)\leq 2\epsilon/3$.

Now we set $l=\lceil 6dt/\epsilon\rceil$. For any $l'\geq l$ , we define $B$ to be the (disjoint) union of all sets $\prod_{j=1}^d\mathbb{I}_{i_j,l'}$ that have a non-empty intersection with $B_1$. Note that $B_1\subseteq B$. Also, for any set $C$ of the form $\{(x_1,\ldots,x_d)\in\mathbb{I}^d:\bigwedge_{i=1}^d\alpha_i\leq x_i\leq \beta_i\}$, the measure of the union of all sets of the type $\prod_{j=1}^d\mathbb{I}_{i_j,l'}$ that are not contained in $C$ and yet have a non-empty intersection with it is at most $2d/l'$. Hence $\lambda(B\setminus B_1)\leq 2dt/l'\leq \epsilon/3$, and so $\lambda(A\Delta B)\leq \epsilon$.
\end{proof}

We make the natural identification of the Lebesgue measure over $\mathbb{I}^d$ with the uniform probability space over this set, so in particular the legitimate events over this probability space are exactly the measurable sets, and for any such set $A$ we identify $\mathrm{Pr}[A]$ with $\lambda(A)$. Another almost immediate consequence is the following.

\begin{lemma}\label{lem:densecube}
If $A$ is a positive probability event over $\mathbb{I}^d$, then for every $\epsilon>0$ there exists a ``box'' set $C=\{(x_1,\ldots,x_d)\in\mathbb{I}^d:\bigwedge_{i=1}^d\alpha_{i}\leq x_i\leq \beta_{i}\}$ (which in particular is also a legitimate event) for which $\mathrm{Pr}[A|C]\geq 1-\epsilon$.
\end{lemma}

\begin{proof}
Similarly to the proof of Lemma \ref{lem:onepix}, we start with a set $B$ such that $A\subset B$, where $B$ is the union of countably many (not necessarily disjoint) ``boxes'' $C_1,C_2,\ldots$ with $C_j=\{(x_1,\ldots,x_d)\in\mathbb{I}^d:\bigwedge_{i=1}^d\alpha_{j,i}\leq x_i\leq \beta_{j,i}\}$, satisfying $\sum_j\lambda(C_j)\leq (1+\epsilon)\lambda(A)$ (we use the assumption that $\lambda(A)>0$ when we pick ``$\epsilon\lambda(A)$'' as the additive term). This in particular means that $\sum_j\lambda(C_j)\leq (1+\epsilon)\sum_j\lambda(A\cap C_j)$, so there exists a particular $j$ for which $\lambda(C_j)\leq (1+\epsilon)\lambda(A\cap C_j)$. Hence $\mathrm{Pr}[A|C_j]=\lambda(A\cap C_j)/\lambda(C_j)\geq 1-\epsilon$ as required.
\end{proof}

We will work with ``combinatorial models'' of relations. For example, a model of an arity $d$ relation over the universe $[n]$ would be a set $R\subseteq [n]^d$, or alternatively a function $R:[n]^d\to\{0,1\}$. A relation of arity $k<d$ can still be represented by a function over $[n]^d$, by making it invariant of the last $n-k$ coordinates.

However, for our purposes (and convenience) we would like to deal with a single function, rather than a function for every relation. To achieve this we allow a larger range, so our main object would be a function $R:[n]^d\to [k]$ for some constant $k$, where for $(i_1,\ldots,i_d)\in [n]^d$ the value $R(i_1,\ldots,i_d)$ provides all information about the existence of this tuple in all relations. For example, for a vocabulary having exactly $r$ relations of arity $d$, we would correspondingly have $k=2^r$, and we would use a correspondence between the values in $[k]$ and the vectors in $\{0,1\}^r$. Thus the single value $R(i_1,\ldots,i_d)$ would tell us which of the relations contain the tuple $(i_1,\ldots,i_d)$.

Another important point is that we maintain the natural order over $[n]$ as a fixed order of the universe. The bulk of this work is about transforming a model in a way that maintains sentences with universal quantifiers, which translates to transforming a model in a way that does not add any new order-preserving ``substructures'' that did not exist in the original model.

We will deal with the scenario where the universe is not finite, and is not even countable. Specifically, we deal with functions $F:\mathbb{I}^d\to [k]$. The main theme of this work is to show that these can be approximated by ``essentially finite'' models, but first we need more clarifications.

We will only deal with measurable functions, that is, functions $F:\mathbb{I}^d\to [k]$ for which the preimage set $F^{-1}(i)$ is Lebesgue-measurable for every $i\in [k]$. In particular, for the uniform probability distribution over $\mathbb{I}^d$, every set $F^{-1}(i)$ corresponds to a legitimate probabilistic event ``$F(x_1,\ldots,x_d)=i$''. We will use the Hamming distance between functions.

\begin{defi}
For two measurable functions $F,G:\mathbb{I}^d\to [k]$, the {\em distance} between them is $d(F,G)=\mathrm{Pr}[F(x_1,\ldots,x_d)\neq G(x_1,\ldots,x_d)]$, where $(x_1,\ldots,x_d)$ is drawn uniformly from $\mathbb{I}^d$. Alternatively we can write $d(F,G)=\lambda(A)$, where $A=\{(x_1,\ldots,x_d)\in\mathbb{I}^d:F(x_1,\ldots,x_d)\neq G(x_1,\ldots,x_d)\}$ is the set of differences.
\end{defi}

Given a function $F:\mathbb{I}^d\to [k]$, we can go back and pick functions $R:[n]^d\to [k]$ by picking values $0<x_1<\cdots<x_n\leq 1$ for the coordinates, and then defining $R$ by the appropriate restriction. By using a uniformly random choice for $0<x_1<\cdots<x_n\leq 1$ we obtain a probability distribution over these functions.

\begin{defi}\label{def:stat}
Given a measurable function $F:\mathbb{I}^d\to [k]$, the {\em $[n]^d$-statistic distribution} is the (finite) probability distribution $\mu_{F,n}$ over functions $R:[n]^d\to [k]$ that results from picking $0<x_1<\cdots<x_n\leq 1$ uniformly from the set of all such sequences, and then setting $R(i_1,\ldots,i_d)=F(x_{i_1},\ldots,x_{i_d})$ for every $(i_1,\ldots,i_d)\in [n]^d$.
\end{defi}

Note that in the above definition, a uniform choice of $0<x_1<\cdots<x_n\leq 1$ can be obtained by picking $(y_1,\ldots,y_n)\in\mathbb{I}^n$ uniformly (and independently) at random, and then setting $x_j$ to be the $j$'th smallest value among $y_1,\ldots,y_n$ for every $j\in [n]$ (here and in other places we ignore the probability zero event that $y_j=y_{j'}$ for some $1\leq j<j'\leq n$).

We also define what it means for a certain $R$ to appear in $F$, with or without positive probability.

\begin{defi}\label{defi:app}
For a measurable function $F:\mathbb{I}^d\to [k]$ and a function $R:[n]^d\to [k]$, we say that {\em $R$ appears in $F$} if there exist $0<x_1<\cdots<x_n\leq 1$ so that $F(x_{i_1},\ldots,x_{i_d})=R(i_1,\ldots,i_d)$ for every $(i_1,\ldots,i_d)\in [n]^d$. We say that {\em $R$ appears in $F$ with positive probability} if we additionally have $\mu_{F,n}(R)>0$, where $\mu_{F,n}$ is the $[n]^d$-statistic probability distribution.
\end{defi}

For the analysis, we will also define an appearance of $R$ in a discrete $S:[m]^d\to [k]$ (in this setting one should think of all appearances as being with positive probability).

\begin{defi}\label{defi:intapp}
For a (discrete) function $S:[m]^d\to [k]$ and a function $R:[n]^d\to [k]$, we say that {\em $R$ appears in $S$} if there exist $1\leq j_1<\cdots<j_n\leq m$ so that $S(j_{i_1},\ldots,j_{i_d})=R(i_1,\ldots,i_d)$ for every $(i_1,\ldots,i_d)\in [n]^d$.
\end{defi}

\section{Presentation of the main result}\label{sec:present}

We will be particularly interested in functions $F:\mathbb{I}^d\to [k]$ that ``do not use the infiniteness of $\mathbb{I}$''. The definition follows.

\begin{defi}\label{def:phomo}
A function $F:\mathbb{I}^d\to [k]$ is called {\em $l$-part homogeneous} if $F(x_1,\ldots,x_d)$ depends only on $\lceil lx_1 \rceil,\ldots,\lceil lx_d \rceil$ and on the order of $x_1,\ldots,x_d$ (i.e.\ whether $x_j\leq x_{j'}$ and/or $x_j\geq x_{j'}$ for every $1\leq j<j'\leq d$).
\end{defi}

In other words, given $(i_1,\ldots,i_d)\in [l]^d$, the values of $F(x_1,\ldots,x_d)$ inside $\prod_{j=1}^d\mathbb{I}_{i_j,l}$ depend only on the order of $x_1,\ldots,x_d$. Note also that whenever $i_j<i_{j'}$ we have $x_j<x_{j'}$ unconditionally, so the order of $x_1,\ldots,x_d$ is non-determined only when we have $j\neq j'$ for which $i_j=i_{j'}$. The following two observations are quite immediate.

\begin{obs}
If $F:\mathbb{I}^d\to [k]$ is $l$-part homogeneous, then every $R:[n]^d\to [k]$ that appears in $F$, appears in it with positive probability, and in particular $\mu_{F,n}(R)\geq 1/l^nn!$. \qed
\end{obs}

\begin{obs}\label{obs:fin}
For every $l$, $d$ and $k$ there is only a finite number of possible $l$-part homogeneous functions $F:\mathbb{I}^d\to [k]$. \qed
\end{obs}

It is almost immediate from Lemma \ref{lem:onepix} that measurable functions can be approximated by $l$-part homogeneous functions for $l$ large enough (that may depend on the function itself).

\begin{lemma}\label{lem:allpix}
For every measurable function $F:\mathbb{I}^d\to [k]$ and $\epsilon>0$ there exists $t$ (that may depend on $F$ and $\epsilon$), so that for every $l\geq t$ there exists an $l$-part homogeneous function $G:\mathbb{I}^d\to [k]$ for which $d(F,G)\leq\epsilon$.
\end{lemma}

\begin{proof}
In fact we will prove something stronger, and approximate $F$ by a function $G$ that is completely constant over any set of the type $\prod_{j=1}^d\mathbb{I}_{i_j,l}$. Denoting the preimage sets by $A_i=F^{-1}(i)$ for every $1\leq i\leq k$, we set $t=\max\{l_1,\ldots,l_k\}$ where $l_i$ is provided by Lemma \ref{lem:onepix} given $A_i$ (as $A$) with $\epsilon/k$ instead of $\epsilon$. Then we use Lemma $\ref{lem:onepix}$ for every $A_i$ with $l\geq t\geq l_i$ to obtain $B_i$ for which the measure of $B_i\Delta A_i$ is at most $\epsilon/k$.

To define $G$, we define the function $R:[l]^d\to [k]$ for which $G(x_1,\ldots,x_d)=R(i_1,\ldots,i_d)$ whenever $x_j\in\mathbb{I}_{i_j,l}$ for every $j\in [d]$. For every $(i_1,\ldots,i_d)\in [l]^d$, we look at the set $\{i:\prod_{j=1}^d\mathbb{I}_{i_j,l}\cap B_i\neq\emptyset\}$. If this is a set of size one, we define $R(i_1,\ldots,i_d)$ to be equal to its sole member, and otherwise we take an arbitrary value for $R(i_1,\ldots,i_d)$. Finally, we note that if $F(x_1,\ldots,x_d)\neq G(x_1,\ldots,x_d)$, then there exists some $i$ for which $(x_1,\ldots,x_d)\in A_i\Delta B_i$ (otherwise, since the sets $A_i\cap B_i$ are all disjoint, $(x_1,\ldots,x_d)$ would have belonged to exactly one of them). This implies that $\lambda(\{(x_1,\ldots,x_d):F(x_1,\ldots,x_d)\neq G(x_1,\ldots,x_d)\})\leq\sum_{i=1}^k\lambda(A_i\Delta B_i)\leq\epsilon$, as required.
\end{proof}

The main result here states that a measurable function can not only be approximated by a homogeneous one, but that this can be done in a way that does not introduce any ``new artifacts'' into this function, in terms of which $R:[n]^d\to [k]$ appear in it (for all $n$ at once).

\begin{theorem}\label{thm:pix}
For every measurable function $F:\mathbb{I}^d\to [k]$ and $\epsilon>0$ there exists an $l$-part homogeneous function $G:\mathbb{I}^d\to [k]$ (for some $l$ that may depend on $F$ and $\epsilon$) for which $d(F,G)\leq\epsilon$, and furthermore satisfying that every function $R:[n]^d\to [k]$ (for any natural number $n$) that appears in $G$ already appears with positive probability in $F$.
\end{theorem}

It would have been nice to show this without the dependency on the order of $x_1,\ldots,x_n$ (see Definition \ref{def:phomo}), but this cannot be done. Consider for example $F:\mathbb{I}^2\to [3]$ that is defined by $F(x_1,x_2)=1$ if $x_1<x_2$, $F(x_1,x_2)=2$ if $x_1=x_2$, and $F(x_1,x_2)=3$ if $x_1>x_2$. Any $l$-part homogeneous function $G$ that has no dependency on the order of $x_1$ and $x_2$ would have a positive probability appearance of some $R:[2]^2\to [3]$ which is completely constant, while for $F$ such an $R$ does not exist (the all-$2$ function ``appears'' in $F$ only if we allow equality of the coordinates, $x_1=x_2$, and even then it appears with zero probability).

\section{Homogeneous functions and inlays}\label{sec:homo}

For the proof of Theorem \ref{thm:pix} we will analyze discrete versions of homogeneous functions, and their appearance with positive probability in the original function. The following is their definition.

\begin{defi}\label{def:intphomo}
For $l$ and $s\geq d$, a function $R:[sl]^d\to [k]$ is called {\em $l$-part homogeneous} if $R(i_1,\ldots,i_d)$ depends only on $\lceil i_1/l \rceil,\ldots,\lceil i_d/l \rceil$ and on the order of $i_1,\ldots,i_d$ (i.e.\ whether $i_j\leq i_{j'}$ and/or $i_j\geq i_{j'}$ for every $1\leq j<j'\leq d$).
\end{defi}

The reason for requiring that $s\geq d$ is so that all orders will be ``expressed'' in every interval. For example, even if $i_1,\ldots,i_d$ all satisfy $\lceil i_j/s\rceil=1$, requiring that $s\geq d$ makes it still possible to have any order between them (and in particular they can be all unequal). This provides soundness to the following definition of compatibility, which means that two homogeneous functions with different domains (both being discrete, or one of them being continuous) have the same ``big picture''.

\begin{defi}\label{def:compat}
Two $l$-part homogeneous functions $R:[sl]^d\to [k]$ and $S:[tl]^d\to [k]$ are called {\em compatible} if for every $(i_1,\ldots,i_d)\in [sl]^d$ and $(i'_1,\ldots,i'_d)\in [tl]^d$, that satisfy $\lceil i_j/s \rceil=\lceil i'_j/t\rceil$ for all $j\in [d]$ as well as that $i_j\leq i_{j'}$ if and only if $i'_j\leq i'_{j'}$ for every $j\neq j'$, the equality $R(i_1,\ldots,i_d)=S(i'_1,\ldots,i'_d)$ holds.

An $l$-part homogeneous function $R:[sl]^d\to [k]$ and an $l$-part homogeneous function $F:\mathbb{I}^d\to [k]$ are called {\em compatible} if for every $(i_1,\ldots,i_d)\in [sl]^d$ and $(x_1,\ldots,x_d)\in \mathbb{I}^d$, that satisfy $\lceil i_j/s \rceil=\lceil lx_j\rceil$ for all $j\in [d]$, as well as that $i_j\leq i_{j'}$ if and only if $x_j\leq x_{j'}$ for every $j\neq j'$, the equality $R(i_1,\ldots,i_d)=F(x_1,\ldots,x_d)$ holds.
\end{defi}

The following observation in particular would not have been true without the requirement that $s\geq d$ in Definition \ref{def:intphomo}.

\begin{obs}\label{obs:onecomp}
If $R:[sl]^d\to [k]$ is $l$-part homogeneous (and $s\geq d$), then for every $t\geq d$ there is exactly one $l$-part homogeneous function $S:[tl]^d\to [k]$ that is compatible with $R$. Additionally, there is exactly one $l$-part homogeneous $F:\mathbb{I}^d\to [k]$ that is compatible with $R$, and for any $l$-part homogeneous $F:\mathbb{I}^d\to [k]$ there is exactly one $l$-part homogeneous $S:[tl]^d\to [k]$ that is compatible with $F$.
\end{obs}

\begin{proof}[Proof sketch]
As an example, here is how the first statement is proved. Given $R:[sl]^d\to [k]$ that is $l$-part homogeneous and $S:[tl]^d\to [k]$ that is compatible with $R$, consider $(i'_1,\ldots,i'_d)\in [tl]^d$. Since $s\geq d$, it is possible to find $(i_1,\ldots,i_d)\in [sl]^d$ that satisfy $\lceil i_j/s \rceil=\lceil i'_j/t\rceil$ for all $j\in [d]$ as well as that $i_j\leq i_{j'}$ if and only if $i'_j\leq i'_{j'}$ for every $j\neq j'$.

To find $(i_1,\ldots,i_d)$, we set $a_j=\lceil i'_j/t\rceil$ and $b'_j=i'_j-(a_j-1)t$ for $j\in [d]$. We now set $i_j=(a_j-1)s+b_j$, where $b_j=|\{j'\in [d]:b'_{j'}<b'_j\}|+1$. This ensures that $b_1,\ldots,b_d$ have the same order relations between them as $b'_1,\ldots,b'_d$ do. Now $\lceil i_j/s\rceil=\lceil i'_j/t\rceil$ for all $j\in [d]$ since $b_j\in [d]$ and $t\geq d$. Finally this implies that $i_1,\ldots,i_d$ and $i'_1,\ldots,i'_d$ have the same order between them by the above guarantee for $b_1,\ldots,b_d$ and $b'_1,\ldots,b'_d$.

Thus $S(i'_1,\ldots,i'_d)=R(i_1,\ldots,i_d)$ by Definition \ref{def:compat}, and in particular there is only one possible value for $S$ in this location. Since we had no prior restrictions on $(i'_1,\ldots,i'_d)$ apart from belonging to $[tl]^d$, this determines the entirety of $S$.
\end{proof}

A crucial part of the proof of the main result requires proving the existence of substructures that also respect certain ``boundaries'', as per the following definition.

\begin{defi}\label{def:inlay}
For a function $R:[sl]^d\to [k]$ and a function $F:\mathbb{I}^d\to [k]$, we say that $R$ is an {\em $s$ over $[l]^d$ inlay of $F$} if for every $a\in [l]$ there exist $0< x_{a,1}<\cdots<x_{a,s}\leq 1$, so that for every $(i_1,\ldots,i_d)\in [sl]^d$, denoting $a_j=\lceil i_j/s \rceil$ and $b_j=i_j-(a_j-1)s$, we have $R(i_1,\ldots,i_d)=F((a_1-1+x_{a_1,b_1})/l,\ldots,(a_d-1+x_{a_d,b_d})/l)$.
\end{defi}

In other words, an $s$ over $[l]^d$ inlay of $F$ is its ``restriction to a subgrid'' that results from considering the partition $\mathbb{I}_{1,l},\ldots,\mathbb{I}_{l,l}$ of $\mathbb{I}$, and selecting a set of size $s$ from every interval.

We also define inlays of discrete functions $S:[tl]^d\to [k]$ for $t\geq s$, this time by considering the partitioning of $[tl]$ to $l$ ``intervals'' of size $[t]$.

\begin{defi}\label{def:intinlay}
For a function $R:[sl]^d\to [k]$ and a function $S:[tl]^d\to [k]$, where $s\leq t$, we say that $R$ is an {\em $s$ from $t$ over $[l]^d$ inlay of $S$} if for every $a\in [l]$ there exist $1\leq h_{a,1}<\cdots<h_{a,s}\leq t$, so that for every $(i_1,\ldots,i_d)\in [sl]^d$, denoting $a_j=\lceil i_j/t \rceil$ and $b_j=i_j-(a_j-1)t$, we have $R(i_1,\ldots,i_d)=S((a_1-1)t+h_{a_1,b_1},\ldots,(a_d-1)t+h_{a_d,b_d})$.
\end{defi}

Note the following simple observation.

\begin{obs}
If a function $F:\mathbb{I}^d\to [k]$ is $l$-part homogeneous and $s\geq d$, then every $s$ over $[l]^d$ inlay of $F$ is $l$-part homogeneous and compatible with $F$.

If a function $S:[tl]^d\to [k]$ is $l$-part homogeneous and $t\geq s\geq d$, then every $s$ from $t$ over $[l]^d$ inlay of $S$ is $l$-part homogeneous and compatible with $S$. \qed
\end{obs}

\section{A Ramsey-type lemma}\label{sec:ramsey}

To prove the existence of homogeneous discrete functions, that appear (as inlays) with positive probability in $F:\mathbb{I}^d\to [k]$, we need a form of Ramsey's theorem. The following is Ramsey's theorem for edge-colored (but otherwise simple and non-oriented) hypergraphs, which is well known.

\begin{lemma}[Ramsey's theorem for edge-colored hypergraphs, see \cite{katzramsey}]\label{lem:ramhgraph}
For a set $V$ let $V^{=d}$ denote the set of all subsets of size $d$ of $V$, and let $f:V^{=d}\to A$ be any function whose range is a finite set $A$. There exists a global function $\mathcal{R}_1:\mathbb{N}^3\to\mathbb{N}$ so that if $|V|\geq \mathcal{R}_1(d,|A|,s)$, then there exists $U\subset V$ with $|U|=s$ for which the restriction $f|_{U^{=d}}$ is the constant function with the value $a$ for some $a\in A$.
\end{lemma}

It is not hard to generalize this to the setting where the domain of the function $f$ is over all nonempty subsets of $V$ of size {\em at most} $d$.

\begin{lemma}\label{lem:ramsgraph}
For a set $V$ let $V^{\leq d}$ denote the set of all nonempty subsets of size at most $d$ of $V$, and let $f:V^{\leq d}\to A$ be any function whose range is a finite set $A$. There exists a global function $\mathcal{R}_2:\mathbb{N}^3\to\mathbb{N}$ so that if $|V|\geq \mathcal{R}_2(d,|A|,s)$, then there exists $U\subset V$ with $|U|=s$ for which the restriction $f|_{U^{\leq d}}$ is a function that depends only on the size of the set. That is, there exists $f':[d]\to A$ so that $f(C)=f'(|C|)$ for all $C\in U^{\leq d}$.
\end{lemma}

\begin{proof}
We set $\mathcal{R}_2(d,k,s)=\mathcal{R}_1(1,k,\mathcal{R}_1(2,k,\cdots \mathcal{R}_1(d,k,s)\cdots))$. Given $V$, we employ Lemma \ref{lem:ramhgraph} over $V$ with parameters $1$, $|A|$ and $\mathcal{R}_1(2,k,\cdots \mathcal{R}_1(d,k,s)\cdots)$ to obtain $U_1$ (for the function $f|_{V^{=1}}$), then employ Lemma \ref{lem:ramhgraph} with parameters $2$, $|A|$ and $\mathcal{R}_1(3,k,\cdots \mathcal{R}_1(d,k,s)\cdots)$ over $U_1$ (for $f|_{U_1^{=2}}$) to obtain $U_2$, and so on, until we employ Lemma \ref{lem:ramhgraph} over a set $U_{d-1}$ with parameters $d$, $|A|$ and $s$ (for $f|_{U_{d-1}^{=d}}$) to finally obtain our required set $U=U_d$.
\end{proof}

We will need an even more general version, that holds when there are several ``types'' of vertices, and we need to choose a given number of vertices of every type.

\begin{lemma}\label{lem:rampgraph}
For a set $V$ that is a disjoint union of $l$ sets $V_1,\ldots,V_l$, let $V^{\leq d}$ denote the set of all nonempty subsets of size at most $d$ of $V$, and let $f:V^{\leq d}\to A$ be any function whose range is a finite set $A$. There exists a global function $\mathcal{R}:\mathbb{N}^4\to\mathbb{N}$ so that if $|V_i|\geq \mathcal{R}(l,d,|A|,s)$ for all $i\in [l]$, then there exist $U_1\subset V_1,\ldots,U_l\subset V_l$ with $|U_i|=s$ for $i\in [l]$, for which the restriction $f|_{(U_1\cup\cdots\cup U_l)^{\leq d}}$ is a function that depends only on the sizes of the intersections with $U_1,\ldots,U_l$. That is, there exists $f':P(d,l)\setminus\{(0,\ldots,0)\}\to A$ so that $f(C)=f'(|C\cap U_1|,\ldots,|C\cap U_l|)$ for all $C\in (U_1\cup\cdots\cup U_l)^{\leq d}$, where $P(d,l)$ is the set of all non-negative integer sequences $(d_1,\ldots,d_l)$ that sum up to at most $d$.
\end{lemma}

\begin{proof}
The proof is by induction over $l$, where clearly we can set $\mathcal{R}(1,d,k,s)=\mathcal{R}_2(d,k,s)$. To set $\mathcal{R}(l,d,k,s)$, we consider a set $V$ that is the disjoint union of $V_1,\ldots,V_l$ and a function $f:V^{\leq d}\to A$ with $|A|=k$. We will only use a subset $W$ of $V_l$ of size $k'=\mathcal{R}_2(d,k^{(d+1)^l},s)$ that we choose arbitrarily (the eventual value for $\mathcal{R}(l,d,k,s)$ will be much larger than $k'$).

Next we define a function $f':(V\setminus V_l)^{\leq d}\to A'$ for a corresponding (rather large) range $A'$ that will become clear from the following definitions. For every $C\in(V\setminus V_l)^{=d}$ we will just set $f'(C)=f(C)$. For every $C\in(V\setminus V_l)^{\leq d-1}$, the value $f'(C)$ will be a member of $A\times A^{W^{\leq d-|C|}}$. Specifically, we define $h_C:W^{\leq d-|C|}\to A$ by $h_C(D)=f(C\cup D)$ for every $D\in W^{\leq d-|C|}$, and then define $f'(C)=(f(C),h_C)$.

We now set by induction $\mathcal{R}(l,d,k,s)=\mathcal{R}(l-1,d,k^{k'^d},s)$, and use the induction hypothesis to obtain $U_1\subset V_1,\ldots,U_{l-1}\subset V_{l-1}$, all of size $s$, so that $f'(C)$ for any $C\in (U_1\cup\cdots\cup U_{l-1})^{\leq d}$ depends only on the intersection sizes $|C\cap U_1|,\ldots,|C\cap U_{l-1}|$. Note that this in particular implies the same for $f(C)$, because this value was used for one of the ``coordinates'' of $f'(C)$.

We now work on obtaining $U_l\subset W$. We define $f'':W^{\leq d}\to A''$ by the following. For $C\in W^{=d}$ we just set $f''(C)=f(C)$. For $C\in W^{\leq d-1}$, the value $f''(C)$ will be a member of $A^{P(d-|C|,l-1)}$. We first define $h'_C:P(d-|C|,l-1)\to A$ as follows. By the choice of $U_1,\ldots,U_{l-1}$ with respect to $f'$, for every $(d_1,\ldots,d_{l-1})\in P(d-|C|,l-1)\setminus\{(0,\ldots,0)\}$, if $D,D'\subset U_1\cup\cdots\cup U_{l-1}$ satisfy $|D\cup U_i|=|D'\cup U_i|=d_i$ for all $i\in [l-1]$ then $f(C\cup D)=f(C\cup D')$. We set $h'_C(d_1,\ldots,d_{l-1})$ to this common value of $f$. Additionally (and naturally) we set $h'_C(0,\ldots,0)=f(C)$. Having thus fully defined $h'_C$, we then set $f''(C)=h'_C$.

We now employ Lemma \ref{lem:ramsgraph} to obtain $U_l\subset W\subset V_l$ of size $s$, so that $f''(C)$ depends only on $|C|$ for every $C\in W^{\leq d}$. By this guarantee, along with what we already know about $f(C)$ for $C\in (U_1\cup\cdots\cup U_{l-1})^{\leq d}$, we obtain that $U_1,\ldots,U_l$ are the required sets for the assertion of the lemma.
\end{proof}

We will use a form of Ramsey's theorem that follows from Lemma \ref{lem:rampgraph}, stated in terms of homogeneous inlays of discrete functions.

\begin{lemma}\label{lem:raminlay}
For every $l$, $s\geq d$ and $k$ there exists $r(l,s,d,k)$, so that if $t\geq r$, then every function $S:[tl]^d\to [k]$ contains an $s$ from $t$ over $[l]^d$ inlay that is $l$-part homogeneous.
\end{lemma}

\begin{proof}
This will be by a direct application of Lemma \ref{lem:rampgraph}. Given $l$, $s$, $d$ and $k$, we first define for every $c\in [d]$ the set $B_c$ of functions from $[d]$ onto $[c]$. We then define $A=\bigcup_{c=1}^d A_c$, where $A_c$ is the set of functions from $B_c$ to $[k]$. We finally set $r(l,s,d,k)=\mathcal{R}(l,d,|A|,s)$ where $\mathcal{R}$ is the function of Lemma \ref{lem:rampgraph}.

Given a function $S:[tl]^d\to [k]$ with $t\geq r(l,s,d,k)$, we define $V_i=\{(i-1)t+1,\ldots,it\}$ for $i\in [l]$, and define $f:(\bigcup_{i=1}^lV_i)^{\leq d}\to A$ as follows. Given $C\in (\bigcup_{i=1}^lV_i)^{\leq d}$, we sort the set to obtain $C=\{i_1,\ldots,i_{d'}\}$ with $i_1<\cdots<i_{d'}$ and $d'=|C|\leq d$. We then define $f(C)=g_C$, where $g_C:B_{d'}\to [k]$ is a member of $A_{d'}$. To define $g_C(h)\in [k]$ for an onto function $h:[d]\to [d']$, we set it to the value $S(i_{h(1)},\ldots,i_{h(d)})$.

We now use Lemma \ref{lem:rampgraph} to find sets $U_1\subset V_1,\ldots,U_l\subset V_l$, all of size $s$, so that for any $C\in (U_1\cup\cdots\cup U_l)^{\leq d}$ the value $f(C)$ (which is in fact a function from $B_{|C|}$ to $[k]$) depends only on the intersection sizes $|C\cap U_1|,\ldots,|C\cap U_l|$. To conclude, for every $i\in [l]$ we sort $U_i$ to obtain $(i-1)t+h_{i,1}<\cdots<(i-1)t+h_{i,s}$ (noting that $U_i\subset \{(i-1)t+1,\ldots,it\}$ we obtain $h_{i,j}\in [t]$ for every $i\in [l]$ and $j\in [s]$). Considering the function $R:[sl]^d\to [k]$ which is the $s$ from $t$ over $[l]^d$ inlay of $S$ corresponding (as per Definition \ref{def:intinlay}) to the resulting $h_{i,j}$, we obtain an $l$-part homogeneous function.

To prove that $R$ it is indeed $l$-part homogeneous, consider $(i_1,\ldots,i_d)\in [sl]^d$ and $(i'_1,\ldots,i'_d)\in [sl]^d$ that satisfy $\lceil i_j/s\rceil=\lceil i'_j/s\rceil$ and $i_j\leq i_{j'}$ if and only if $i'_j\leq i'_{j'}$ for every $j,j'\in [d]$. First consider the result of sorting the sets to obtain $c_1<\cdots<c_{d'}$ with $|\{c_1,\ldots,c_{d'}\}|=|\{i_1,\ldots,i_d\}|$ and $c'_1<\cdots<c'_{d'}$ with $|\{c'_1,\ldots,c'_{d'}\}|=|\{i'_1,\ldots,i'_d\}|$, noting that $d'$ is the same for both $c_j$ and $c'_j$ due to the order condition. Due to the same order condition we also have a single onto function $h:[d]\to [d']$ so that $i_j=c_{h(j)}$ and $i'_j=c'_{h(j)}$ for all $j\in [d]$. Finally, noting that this also implies that $\lceil c_j/s\rceil=\lceil c'_j/s\rceil$ for all $j\in [d']$, we have (by the choice of $U_1,\ldots,U_l$ above using Lemma \ref{lem:rampgraph}) a function $g:B_{d'}\to [k]$ for which $f(\{c_1,\ldots,c_{d'}\})=f(\{c'_1,\ldots,c'_{d'}\})=g(h)$. Hence, $R(i_1,\ldots,i_d)=R(i'_1,\ldots,i'_d)=g(h)$ as required.
\end{proof}

There is a way that Ramsey type statements can be converted to probabilistic versions, referring to an event relating to objects (in our case inlays) chosen using an underlying probability distribution. We will use the following version. The reason for using the extra parameters $\alpha_j$ and $\beta_j$ will become clear later on (it will be the result of invoking Lemma \ref{lem:densecube} in the proof of Lemma \ref{lem:appclose} below, on the way towards proving Theorem \ref{thm:pix}).

\begin{lemma}\label{lem:probinlay}
For every $l$, $s\geq d$ and $k$ there exists $\delta(l,s,d,k)>0$ with the following property. Let $F:\mathbb{I}^d\to [k]$ be a measurable function, let $0<\alpha_a<\beta_a\leq 1$ be parameters for $a\in [l]$ (relating to the intervals $\mathbb{I}_{1,l},\ldots,\mathbb{I}_{l,l}$ respectively), and let $R$ be a random $s$ over $[l]^d$ inlay of $f$ chosen in the following manner: For every $a\in [l]$, $\alpha_a\leq x_{a,1}<\cdots<x_{a,s}\leq \beta_a$ are chosen uniformly at random, and $R$ is defined (as per Definition \ref{def:inlay}) by $R(i_1,\ldots,i_d)=F((a_1-1+x_{a_1,b_1})/l,\ldots,(a_d-1+x_{a_d,b_d})/l)$ where $a_j=\lceil i_j/s \rceil$ and $b_j=i_j-(a_j-1)s$. With probability at least $\delta$ the inlay $R$ will be $l$-part homogeneous.
\end{lemma}

\begin{proof}
We consider an alternative way to choose $\alpha_a\leq x_{a,1}<\cdots<x_{a,s}\leq \beta_a$ for $a\in [l]$: For $t=r(l,s,d,k)$ (using the function of Lemma \ref{lem:raminlay}), we first uniformly choose $\alpha_a\leq y_{a,1}<\cdots<y_{a,t}<\beta_a$. Then for every $a\in [l]$ we choose uniformly (from the possible $\binom{t}{s}$ choices) $1\leq j_{a,1}<\cdots<j_{a,s}\leq t$, and set $x_{a,i}=y_{a,j_{a,i}}$ for all $i\in [s]$ and $a\in [l]$.

Now invoking Lemma \ref{lem:raminlay}, we know that for every choice of $\alpha_a\leq y_{a,1}<\cdots<y_{a,t}<\beta_a$, there exist for every $a\in [l]$ some $1\leq h_{a,1}<\cdots<h_{a,s}\leq t$, so that the inlay $R'$ defined by $R'(i_1,\ldots,i_d)=F((a_1-1+y_{a_1,h_{a_1,b_1}})/l,\ldots,(a_d-1+y_{a_d,h_{a_d,b_d}})/l)$ with $a_j=\lceil i_j/s \rceil$ and $b_j=i_j-(a_j-1)s$ is $l$-part homogeneous. Namely, this would be the $l$-part homogeneous $s$ from $t$ over $[l]^d$ inlay of $S$ guaranteed by Lemma \ref{lem:raminlay}, where $S$ is the $t$ over $[l]^d$ inlay of $F$ defined by $S(i_1,\ldots,i_d)=F((a_1-1+y_{a_1,b_1})/l,\ldots,(a_d-1+y_{a_d,b_d})/l)$ with $a_j=\lceil i_j/t \rceil$ and $b_j=i_j-(a_j-1)t$.

We finally set $\delta=1/\binom{t}{s}^l$, the probability that every $j_{a,i}$ is identical to its respective $h_{a,i}$.
\end{proof}

\section{Proof of the main result}\label{sec:proof}

Using Lemma \ref{lem:probinlay} in conjunction with Lemma \ref{lem:densecube}, we can use an $l$-part homogeneous $G:\mathbb{I}^d\to [k]$ that is close to a given $F:\mathbb{I}^d\to [k]$ (but with no other guarantees, such as the one given by Lemma \ref{lem:allpix}), to find an arbitrarily large $l$-part homogeneous discrete function $R:[sl]^d\to [k]$ that appears with positive probability in $F$, and is compatible with some $G':\mathbb{I}^d\to [k]$ that is not extremely further from $F$ as compared to $G$. The following lemma formalizes this. Note that there is a dependency of $G'$ on $s$ in the formulation below, but it will be mitigated later on.

\begin{lemma}\label{lem:appclose}
Suppose that $F:\mathbb{I}^d\to [k]$ is any measurable function and that $G:\mathbb{I}^d\to [k]$ is an $l$-part homogeneous function. For every $s\geq d$, there exists an $l$-homogeneous function $R:[sl]^d\to [k]$ that appears with positive probability in $F$, so that the $l$-homogeneous function $G':\mathbb{I}^d\to [k]$ that is compatible with $R$ (which is unique by Observation \ref{obs:onecomp}) satisfies $d(F,G')\leq 2d(F,G)+\binom{d}{2}/l$.
\end{lemma}

\begin{proof}
We first note that it is enough to show that an $s$ over $[l]^d$ inlay of $F$, chosen at random as per the distribution in Lemma \ref{lem:probinlay} for some $\alpha_1,\ldots,\alpha_l$ and $\beta_1,\ldots,\beta_l$, satisfies the assertions of this lemma with some positive probability $\delta'$. Once we prove this, we note that the above implies that there exists a single function $R:[sl]^d\to [k]$ satisfying the assertion of the lemma that appears with some positive probability $\delta''$, since there is only a finite number of $l$-homogeneous functions $R:[sl]^d\to [k]$. But then, by the chain rule, this means that $\mu_{F,sl}(R)\geq \delta''l^{-sl}\prod_{j=1}^l(\beta_j-\alpha_j)^s>0$.

For the rest of the proof we will show the existence of $\alpha_1,\ldots,\alpha_l$ and $\beta_1,\ldots,\beta_l$ that ensure that the above event happens with $\delta'\geq\delta(l,s,d,k)/2$, where $\delta$ is the function of Lemma \ref{lem:probinlay}. We first analyze a different probability space $\tilde\mu_{F,l}$ over functions $T:[l]^d\to [k]$. This space is defined as the result of choosing uniformly and independently $0<x_i\leq 1$ for every $i\in [l]$, and defining $T$ by $T(i_1,\ldots,i_d)=F((i_1-1+x_{i_1})/l,\ldots,(i_d-1+x_{i_d})/l)$ for $(i_1,\ldots,i_d)\in [l]^d$ (note that this can be viewed as a random ``$1$ over $[l]^d$ inlay''of $F$).

Consider now the set of tuples without repetitions, $I=\{(i_1,\ldots,i_d)\in [l]^d:|\{i_1,\ldots,i_d\}|=d\}$. We analyze probabilistic bounds for the number of $(i_1,\ldots,i_d)\in I$ for which $T(i_1,\ldots,i_d)\neq G((i_1-1+x_{i_1})/l,\ldots,(i_d-1+x_{i_d})/l)$. For this we define $A=\{(x_1,\ldots,x_d)\in\mathbb{I}^d:F(x_1,\ldots,x_d)\neq G(x_1,\ldots,x_d)\}$. For a tuple $(i_1,\ldots,i_d)\in I$, the probability for having $T(i_1,\ldots,i_d)\neq G((i_1-1+x_{i_1})/l,\ldots,(i_d-1+x_{i_d})/l)$ is exactly $l^d\lambda(A\cap \prod_{j=1}^d\mathbb{I}_{i_d,l})$.

Hence, the expected size of the ``set of differences'' $I'=\{(i_1,\ldots,i_d)\in I:T(i_1,\ldots,i_d)\neq G((i_1-1+x_{i_1})/l,\ldots,(i_d-1+x_{i_d})/l)\}$ is at most $\lambda(A)l^d=l^d\cdot d(F,G)$. This means that with positive probability we have $|I'|\leq l^d\cdot d(F,G)$. We use Lemma \ref{lem:densecube} over the choice of $x_1,\ldots,x_l$ with $l$ as the dimension, to obtain $\alpha_1,\ldots,\alpha_l$ and $\beta_1,\ldots,\beta_l$, so that when we condition on $\bigwedge_{j=1}^l\alpha_j\leq x_j\leq \beta_j$, we obtain $|I'|\leq l^d\cdot d(F,G)$ with probability at least $1-\delta(l,s,d,k)/2s^d$.

We now consider the probability space of choosing $\alpha_i\leq x_{a,1}<\cdots<x_{a,s}\leq \beta_i$ uniformly for every $a\in [l]$, and defining $R$ by $R(i_1,\ldots,i_d)=F((a_1-1+x_{a_1,b_1})/l,\ldots,(a_d-1+x_{a_d,b_d})/l)$ where $a_j=\lceil i_j/s \rceil$ and $b_j=i_j-(a_j-1)s$. To analyze further, consider the alternative process of choosing $\alpha_a\leq y_{a,j}\leq \beta_a$ uniformly and independently for $a\in [l]$ and $j\in [s]$, and then sorting $y_{a,1},\ldots,y_{a,s}$ to obtain $x_{a,1}<\cdots<x_{a,s}$ for every $a\in [l]$.

By the choices of $\alpha_1,\ldots,\alpha_l$ and $\beta_1,\ldots,\beta_l$, for every fixed $(j_1,\ldots,j_d)\in [s]^d$, with probability at least $1-\delta(l,s,d,k)/2s^d$ we have $|\{(i_1,\ldots,i_d)\in I:R((i_1-1)s+j_1,\ldots,(i_d-1)s+j_d)\neq G((i_1-1+y_{i_1,j_1})/l,\ldots,(i_d-1+y_{i_d,j_d})/l)\}|\leq d(F,G)\cdot l^d$. By a union bound, with probability at least $1-\delta(l,s,d,k)/2$ all these events happen at once, and then by summing over all $(j_1,\ldots,j_d)\in [s]^d$ we get $|\{(i_1,\ldots,i_d,j_1,\ldots,j_d)\in I\times [s]^d:R((i_1-1)s+j_1,\ldots,(i_d-1)s+j_d)\neq G((i_1-1+x_{i_1,j_1})/l,\ldots,(i_d-1+x_{i_d,j_d})/l)\}|\leq d(F,G)\cdot (sl)^d$ (we used here the fact that $x_{i,j}$ are obtained from respective permutations of $y_{i,j}$). By another union bound, with probability at least $\delta(l,s,d,k)/2$ the above event happens concurrently with $R$ being $l$-part homogeneous.

Now consider the $l$-part homogeneous function $G':\mathbb{I}^d\to [k]$ that is compatible with $R$. Noting the above bound, together with $|[l]^d\setminus I|\leq \binom{d}{2}l^{d-1}$, when the above events all happen we obtain $d(G,G')\leq d(F,G)+\binom{d}{2}/l$, and hence by the triangle inequality $d(F,G')\leq 2d(F,G)+\binom{d}{2}/l$.
\end{proof}

We are finally ready for the proof of the main result.

\begin{proof}[Proof of Theorem \ref{thm:pix}]
Given $F:\mathbb{I}^d\to [k]$ and $\epsilon$, we first use Lemma \ref{lem:allpix} to obtain some $G:\mathbb{I}^d\to [k]$ that is $l$-part homogeneous for $l\geq 3\binom{d}{2}/\epsilon$ and satisfies $d(F,G)\leq \epsilon/3$.

We next use Lemma \ref{lem:appclose} to obtain  $G':\mathbb{I}^d\to [k]$ that is compatible with an $l$-part homogeneous $R':[sl]^d\to [k]$ for which $\mu_{F,sl}(R')>0$, and satisfies $d(F,G')\leq 2d(F,G)+\binom{d}{2}/l\leq \epsilon$. We do this for every $s\geq d$, and for every $s$ we may obtain a different $G'$. However, since by Observation \ref{obs:fin} there is a finite number of possible $G'$, we can pick a single $G'$ for which this holds for an infinite sequence of possible $s$.

To complete the proof, it remains to show that no $S:[n]^d\to [k]$ for which $\mu_{F,n}(S)=0$ appears in $G'$. Assuming on the contrary that there exists such an $S$, we note that in particular $S$ also appears in the $l$-part homogeneous $R':[n'l]^d\to [k]$ that is compatible with $G'$ for any $n'\geq n$, and by the above choice of $G'$ we know that there exists such an $n'$ for which $\mu_{F,n'l}(R')>0$. Now recall that we can view the probability space $\mu_{F,n}$ as the result of first choosing $T:[n'l]^d\to [k]$ according to $\mu_{F,n'l}$, then choosing $1\leq j_1<\cdots<j_n\leq n'l$ uniformly (from the $\binom{n'l}{n}$ possible choices), and finally setting $R:[n]\to [k]$ by defining $R(i_1,\ldots,i_d)=T(j_{i_1},\ldots,j_{i_d})$. But this implies that $\mu_{F,n}(S)\geq\mu_{F,n'l}(R')/\binom{n'l}{n}>0$, a contradiction.
\end{proof}

\section{Discussion and variants}\label{sec:disc}

\subsection*{Relation to the original pixelation lemma}

The pixelation lemma in \cite{orderon} was stated specifically for functions $F:\mathbb{U}^4\to \mathbb{U}$, where $\mathbb{U}$ denotes the closed interval $\{x\in\mathbb{R}:0\leq x\leq 1\}$. Also, the definition of appearance is different (there are more ways for a structure to appear in $F$).

Specifically, the domain is interpreted as $(\mathbb{U}\times \mathbb{U})^2$, and $F$ is considered as a {\em binary} relation with ``fractional elements'' (so a single ``vertex'' corresponds to the set $\{a\}\times \mathbb{U}$ for some $0\leq a\leq 1$).  Thus, in the definition of an appearance of a structure (here a vertex-ordered graph), the order relation between the first and the third coordinates is the only one taken into consideration, while for the second and fourth coordinate we are only concerned about whether they can be chosen from a positive probability set (for a positive probability appearance).

Additionally, the range is interpreted as corresponding to a single relation (essentially a vertex-ordered simple graph), corresponding to a function $E:[n]\to\{0,1\}$. The definition of the probability for appearance involves a final step, where for example the value $F((x_1,a),(x_2,b))$ provides the probability that $E(1,2)=1$ (as opposed to $E(1,2)=0$), following the random choice of $0\leq x_1<x_2\leq 1$ and $a,b\in \mathbb{U}$ (independently). Accordingly, the approximation guarantee is phrased in terms of the $L_1$ distance, $d(F,G)=\int_{(x,a,y,b)\in\mathbb{U}^4}|F((x,a),(y,b))-G((x,a),(y,b))|d(x,a,y,b)$.

This can be converted back to functions with discrete ranges by standard ``quantization'', approximating a real value by the closest multiple of $\frac1k$ for some large enough $k$ (and making sure that the ``edge values'' $0$ and $1$ are preserved). Also note that recently in \cite{towsner} it was observed that for the purpose of representing limit objects (including vertex-ordered graphs, and also plain graphs as in the large body of work presented in \cite{graphons}), there is a way of altogether doing away with values outside $0$ and $1$ (in the ordered regime the additional unordered coordinates are still necessary).

Finally, the lemma in \cite{orderon} is restricted to symmetric relations, in the sense that $F((x,a),(y,b))=F((y,b),(x,a))$ for every $x,a,y,b\in\mathbb{U}$. This is not an essential difference, since a non-symmetric relation can be represented by two symmetric relations as long as we have an underlying vertex order. The way Lemma \ref{lem:rampgraph} is converted to Lemma \ref{lem:raminlay} in fact demonstrates how such representations can be constructed for relations of any arity. However, after such a conversion we must still distinguish which of the coordinates are equal to each other, and so must explicitly handle relations of lower arities. Specifically in \cite{orderon} the role of equality is also very diminished, since there the values $F((x,a),(y,a))$ for $x=y$ are ignored by sticking to the notion of graphs without loops.

The special case pixelation lemma from \cite{orderon} is used there in conjunction with a compactness theorem for the above described limit objects (under a suitable topology, weaker than that of the $L_1$ distance) to derive a removal lemma for vertex-ordered graphs. While the pixelation lemma is generalized to higher arities here, the original concept of limit objects does not generalize that easily. Investigations of limit objects for hypergraphs are present in \cite{elekhyp} and \cite{zhaolim}, and with the addition of a vertex order in \cite{towsner}. There, a replacement for the pixelation lemma is used that still guarantees its most useful property in that context, namely the assurance that all appearances in the converted function are of positive probability and come from positive probability appearances in the original function.

\subsection*{Dealing with diagonals and lower arity relations}

Recall that for $k=2^r$, a function $F:\mathbb{I}^d\to [k]$ can give the information about a model (over $\mathbb{I}$) of $r$ arity $d$ relations. Recall also the earlier comment that lower arity relations can be represented by making them invariant over the last coordinates. Thus an arity $d'$ relation for $d'<d$ can be replaced by an arity $d$ relation, if we stipulate that for every $x_1,\ldots,x_d$ and $x'_{d'+1},\ldots,x'_d$ we have the equivalence $R(x_1,\ldots,x_d)\leftrightarrow R(x_1,\ldots,x_{d'},x'_{d'+1},\ldots,x'_d)$.

When we move to modeling with a single function $F$, this condition can be converted to stipulating that certain structures do not appear. If for example the relation in question is the first relation in the vocabulary, meaning that it is represented by $F(x_1,\ldots,x_d) \pmod{2}$, then the additional ``forbidden structures'' are all $S:[2]^d\to [k]$ for which $S(1,\ldots,1)\neq S(1,\ldots,1,i_{d'+1},\ldots,i_d) \pmod{2}$ for any $i_{d'+1},\ldots,i_d$ that take values in $\{1,2\}$.

Thus, the $l$-part homogeneous $G$ that results from Theorem \ref{thm:pix} will also satisfy the condition that makes it conform to a relation of arity $d'$. Additionally, the way this relation is modeled (being invariant of the last $d-d'$ coordinates), a distance bound between $F$ and $G$ in terms of the Lebesgue measure over $\mathbb{I}^d$ translates to a corresponding distance bound that applies to the original relation in terms of the measure over $\mathbb{I}^{d'}$.

It may at times be useful to also ensure for an arity $d$ relation that measures are preserved for the restrictions to ``diagonals'', i.e., when there are equality constraints between coordinates. The logical equivalent is when a relation is used with the same variable appearing in more than one place. For example, when we are dealing with a binary relation, one might want to ensure a small distance also with respect to the measure of the set $\{x\in\mathbb{I}:F(x,x)\neq G(x,x)\}$, where $F$ is the function referring to $R$ and other relations.

The way to ensure such a small distance bound is by constructing a unary relation $U(x)$, and adding the condition $U(x)\leftrightarrow R(x,x)$ for all $x$. This can be converted to a condition about certain substructures not appearing in $F$. The relation $U$ can then be converted back to a binary relation as explained above and added to the encoding by $F$, to make sure that the small distance bound from $G$ applies to it.

\subsection*{Removing the order dependency at a cost}

Recall that by the comment following the statement of Theorem \ref{thm:pix}, to ensure the exclusion of structures that do not appear in the original function $F$ it is necessary that the $l$-part homogeneous function $G:\mathbb{I}^d\to [k]$ has a dependency on the order between the coordinate values. This is relevant in sets of type $\prod_{j=1}^l\mathbb{I}_{i_j,l}$ whenever $i_1,\ldots,i_l$ are not all different. If we insist that we want a ``completely pixelated'' $G$, which is constant over all sets of the type $\prod_{j=1}^l\mathbb{I}_{i_j,l}$, we have to alter the exclusion requirement.

For example, suppose that we have a single relation $R$ of arity $2$ and we are interested in substructures with $2$ elements. Then we would look at the quartet $(R(x,x),R(x,y),R(y,x),R(y,y))$ for any $x<y$. But if we count ``homomorphisms'' as structures to be excluded as well, we would also look at $(R(x,x),R(x,x),R(x,x),R(x,x))$, corresponding to the case $x=y$, and consider the measure of the set of $x\in\mathbb{I}$ that provide a certain value.

If ``substructures with equalities'' are also considered as substructures that can appear with positive probability in $F$, then we can have a completely pixelated $H$ with the following procedure: We start with the $G$ provided by Theorem \ref{thm:pix}, but then for every $(i_1,\ldots,i_d)\in [l]^d$ for which some indexes are equal, we replace the values of $G$ over ${\prod_{j=1}^l\mathbb{I}_{i_j,l}}$ with the constant equal to $G((i_1+\frac12)/l,\ldots,(i_d+\frac12)/l)$. This is the same as ``stipulating'' that whenever $x_j$ and $x_j'$ {\em can} be equal (since $i_j=i_j'$), they {\em must} be equal.

The new structures appearing in $H$ might not appear according to the original definition of appearance in $F$, but they must all appear in $F$ when we allow equalities among $x_1,\ldots,x_n$ in the definition of appearance. In fact they appear with positive probability when we condition the $[n]^d$-statistic distribution $\mu_{F,n}$ on the event of the respective equalities occurring among $x_1,\ldots,x_n$.

The original lemma in \cite{orderon} does not have an order dependency, but this is a benefit of dealing only with symmetric binary structures without loops (which correspond to equalities). The version of Ramsey's theorem used in its proof there is also much lighter than the one developed here. If we only ignore orders, for example allowing only the hard-coded relations ``$=$'' and ``$\neq$'' instead of ``$\leq$'', then we would obtain a lemma where the value of $F:(x_1,\ldots,x_d)$ depends only on the partition of $x_1,\ldots,x_d$ into parts with equal values. For arities larger than $2$, its proof would still require the version of Ramsey's theorem developed here.

\subsubsection*{Containment in the other direction}

Considering the function $F:\mathbb{I}^d\to [k]$, Theorem \ref{thm:pix} ensures the existence of an $l$-part homogeneous $G:\mathbb{I}^d\to [k]$ (for some $l$) within distance $\epsilon$ of $F$, so that all structures that appear in $G$ already appear in $F$. One can ask whether this can be made bidirectional, so that $G$ will also be guaranteed to contain every structure that appears with positive probability in $F$.

However, this does not hold even for the (rather non-interesting) case of $d=1$. For every natural number $r$, define $\mathbb{J}_r=\mathbb{I}_{2^r,2}$. Note that $\mathbb{J}_1,\mathbb{J}_2,\ldots$ are all disjoint and their union equals $\mathbb{I}$. Then define $F:\mathbb{I}\to\{1,2\}$ by setting $F(x)=1$ if $x\in \mathbb{J}_{r}$ for an even $r$, and setting $F(x)=2$ if $x\in \mathbb{J}_{r}$ for an odd $r$. In this construction, for every $r$ there exists $x_1<\cdots<x_r$ so that $F(x_i)=1$ if and only if $i$ is odd. However, for every $l$, an $l$-part homogeneous function $G:\mathbb{I}\to\{1,2\}$ will not contain such a sequence for any $r>l$.

On the other hand, for every fixed $r$, one can still ensure that the $l$-part homogeneous $G:\mathbb{I}^d\to [k]$ resulting from Theorem \ref{thm:pix} contains all structures of size $r$ that appear with positive probability in $F$. To this end, let $\delta$ be the minimum of $\mu_{F,r}(R)$ over all $R:[r]^d\to [k]$ for which $\mu_{F,r}(R)>0$. Then, deploy Theorem \ref{thm:pix}, replacing the original parameter $\epsilon$ with $\min\{\epsilon,\delta/2r^d\}$.

\bibliographystyle{amsplain}
\bibliography{pixel}

\end{document}